\documentclass[12pt, reqno]{amsart}
\usepackage{amsmath, amsthm, amscd, amsfonts, amssymb, graphicx, color, mathrsfs}
\usepackage[bookmarksnumbered, colorlinks, plainpages]{hyperref}

\newtheorem{theorem}{Theorem}[section]
\newtheorem{lemma}[theorem]{Lemma}

\theoremstyle{definition}

\newtheorem{example}[theorem]{Example}

\theoremstyle{remark}
\newtheorem{remark}[theorem]{Remark}
\numberwithin{equation}{section}

\begin{document}
\setcounter{page}{1}

\title[H\"older-Besov boundedness for pseudo-differential operators]{ H\"older-Besov boundedness for periodic pseudo-differential operators }

\author[Duv\'an Cardona S\'anchez]{Duv\'an Cardona $^1$}

\address{$^{1}$ Department of Mathematics, Universidad de los Andes, Colombia.}
\email{\textcolor[rgb]{0.00,0.00,0.84}{duvanc306@gmail.com;
d.cardona@uniandes.edu.co}}


\subjclass[2010]{}


\date{Received: xxxxxx;  Revised: yyyyyy; Accepted: zzzzzz.
\newline \indent $^{1}$ Universidad de los Andes, Mathematics Department, Bogot\'a - Colombia}

\begin{abstract}
In this  work we give H\"older-Besov estimates for periodic Fourier multipliers. We present a class of bounded  pseudo-differential operators on periodic Besov spaces with symbols of limited regularity.
\end{abstract} \maketitle

\section{Introduction}

In this paper we study the boundedness of periodic Fourier multipliers and periodic pseudo-differential operators from H\"older spaces into Besov spaces. Let $\sigma:\mathbb{Z}\rightarrow \mathbb{C}$ be a {\em symbol}, the corresponding Fourier multiplier $\text{Op}(\sigma)$ is the periodic pseudo-differential operator  formally defined by the formula
\begin{equation}\label{mt}
\text{Op}(\sigma(\cdot))f=\mathscr{F}^{-1}(\sigma(\xi)\mathscr{F}(f)),
\end{equation}
where $\mathscr{F}$ is the Fourier transform on the torus $\mathbb{T}=[0,2\pi)$ and $\mathscr{F}^{-1}$ is the inverse Fourier transform. In 1979, Agranovich  \cite{ag} proposed a global quantization of periodic pseudo-differential operators on the circle $\mathbb{S}^{1}\equiv\mathbb{T}$. Later, this theory was widely developed by Ruzhansky and Turunen in \cite{Ruz}, where the theory of periodic pseudo-differential operators is considered in arbitrary dimensions. Periodic Besov spaces form a class of function spaces  which are of special interest in analysis and mathematical physics. They can be defined via dyadic decomposition and form scales $B^r_{p,q}(\mathbb{T})$ carrying three indices: $r\in\mathbb{R}$, $0<p,q\leq \infty.$ In the special case $p=q=\infty,$ $\Lambda^r(\mathbb{T})=B^r_{\infty,\infty}(\mathbb{T})$ is nothing else but the familiar space of all H\"older continuous functions of order $0<r<1.$  There are several possibilities concerning the conditions to impose on a symbol $\sigma$ in the attempt to establish a periodic Fourier multiplier theorem of boundedness on Besov spaces and Lebesgue spaces for its corresponding operator (\ref{mt}) (see \cite{Bar, Bie, Bu1, Bu2, Bu3}). In this paper we investigate the action of periodic Fourier multipliers and periodic pseudo-differential operators from H\"older spaces into Besov spaces. Our work is closely related with a classical result by Marcinkiewicz: if $(\sigma(\xi))_{\xi\in\mathbb{Z}}$ is a sequence satisfying the following condition, now  known as variational Marcinkiewicz condition:
\begin{equation}\label{mc}
\Vert \sigma \Vert_{L^{\infty}(\mathbb{Z})}+\sup_{j\geq 0}\sum_{2^j\leq |\xi|\leq 2^{j+1}}|\sigma(\xi+1)-\sigma(\xi)|<\infty,
\end{equation}
then $\text{Op}(\sigma):L^p(\mathbb{T})\rightarrow L^p(\mathbb{T}) $ is a bounded operator for all $1<p<\infty.$ Here one may consider $\Delta \sigma(\cdot)=\sigma(\cdot+1)-\sigma(\cdot)$ as the first derivative of $\sigma.$ As a particular case of Theorem 4.2 in \cite{Ar}, every operator $\text{Op}(\sigma)$ satisfying  \eqref{mc} is a bounded operator from $B^r_{p,q}(\mathbb{T})$ into $B^r_{p,q}(\mathbb{T})$ for all $1<p<\infty,$ $r\in\mathbb{R}$ and $1\leq q\leq \infty.$ We observe that, by Corollary 4.3 in \cite{Ar}, for every $r\in(0,1)$ there exists a Fourier multiplier $\text{Op}(\sigma)$ with $\sigma$ satisfying \eqref{mc}, but with the property that $\text{Op}(\sigma)$ is not a Fourier multiplier from $B^r_{\infty,\infty}(\mathbb{T})$ into  $B^r_{\infty,\infty}(\mathbb{T})$. In order to get, in particular, boundedness of periodic Fourier multipliers on H\"older spaces, we reformulate the variational Marcinkiewicz condition by imposing the following inequality on the symbol:
\begin{equation}\label{rho2}
|\sigma(\xi)|\leq C|\xi|^{-\rho},
\end{equation}
uniformly on $\xi\neq 0,$ for some $0\leq \rho\leq 1.$ Later, by using estimates on Fourier multipliers, we deduce the boundedness of operators with symbols $\sigma(x,\xi)$ of finite regularity on $x.$ More precisely, symbols satisfying inequalities of the type
\begin{equation}\label{rho}
|\Delta_{\xi}^{\alpha}\partial_{x}^{\beta}\sigma(x,\xi)|\leq C_\beta |\xi|^{-\rho-|\alpha|},
\end{equation}
for $|\alpha|\leq l_1,|\beta|\leq l_2,$ $l_{i}<\infty.$ We note that, condition \eqref{rho} is related with the H\"ormander class of symbols on the torus proposed by Ruzhansky and Turunen in \cite{Ruz}. In Section \ref{mrp} we show that, under suitable conditions on the set of indices $p, q, r, s$ and $\rho,$ the $\rho$-condition \eqref{rho} implies the boundedness of $\text{Op}(\sigma(\cdot))$ from $B^s_{\infty,\infty}(\mathbb{T})$ into $B^r_{p,q}(\mathbb{T}),$ then we extend these results to the case of  pseudo-differential operators on the torus. We end Section \ref{mrp} with a discussion of our main results and some applications.\\

\noindent Finally, let us give some references on the topic we use along this paper. The boundedness of Fourier multipliers in $L^p$-spaces, H\"older spaces and Besov spaces has been considered by many authors for a long time. In the general case of Compact Lie groups we refer the reader to the  works of Alexopoulos, Anker, Coifman, Ruzhansky, Turunen and Wirth  \cite{Al, An, Co3, Ruz, Ruz-2, Ruz-T, RWT, Ruz-3}. The general case of operator-valued Fourier multipliers on the torus has been investigated by Arendt, Bu, Barraza, Denk, Hern\'andez, and Nau  in \cite{Ar, Bar, Bie, Bu1, Bu2, Bu3}. $L^p$ and H\"older estimates of periodic pseudo-differential operators can be found in \cite{Duvan1, Duvan2, Duvan3,  Delg} and \cite{s2}. The quantization process, $L^2$-compactness, spectral properties and $L^p$ estimates of pseudo-differential operators on the circle $\mathbb{S}^1\equiv\mathbb{T}$ also can be found in the works of Delgado, Wong and Molahajloo  \cite{DW, s1, s11, s2, MW}. Besov continuity of Fourier multipliers and pseudo-differential operators on general compact Lie groups has been investigated by the author in \cite{Duvan4}.

\section{preliminaries}\label{preliminaries}

We use the standard notation of pseudo-differential operators (see e.g. \cite{Ruz}). The Schwartz space $\mathcal{S}(\mathbb{Z}^n)$ denote the space of functions $\phi:\mathbb{Z}^n\rightarrow \mathbb{C}$ such that 
 \begin{equation}
 \forall M\in\mathbb{R}, \exists C_{M}>0,\, |\phi(\xi)|\leq C_{M}\langle \xi \rangle^M,
 \end{equation}
where $\langle \xi \rangle=(1+|\xi|^2)^{\frac{1}{2}}.$ The toroidal Fourier transform is defined for any $f\in C^{\infty}(\mathbb{T}^n)$ by $$(\mathscr{F}f)(\xi):=\widehat{f}(\xi)=\int_{\mathbb{T}^n}e^{-i\langle x,\xi\rangle}f(x)dx,\,\,\xi\in\mathbb{Z}^n,$$ 
where $dx$ is the Haar measure on the $n$-torus $\mathbb{T}^n=[0,2\pi)^n.$ The inversion formula is given by $$f(x)=\sum_{\xi\in\mathbb{Z}^n}e^{i\langle x,\xi \rangle }\widehat{u}(\xi),\,\, x\in\mathbb{T}^n.$$

We now take up the H\"older space $\Lambda^s,$ $0<s<1.$ According to the usual definition, a function $f$ belongs to $\Lambda^s$ if there exists a constant $A$ so that $|f(x)|\leq A$ almost every where and 
\begin{equation}|f|_{\Lambda^s}:=\sup_{x,y}\frac{|f(x-y)-f(x)|}{|y|^{s}}\leq A.
\end{equation}

We introduce the Besov spaces on the torus using the periodic Fourier transform as follow. Let $r\in\mathbb{R},$ $0\leq q<\infty$ and $0<p\leq \infty.$ If $f$ is a measurable function on $\mathbb{T},$ we say that $f\in B^r_{p,q}(\mathbb{T})$ if $f$ satisfies
\begin{equation}\Vert f \Vert_{B^r_{p,q}}:=\left( \sum_{m=0}^{\infty} 2^{mrq}\Vert \sum_{2^m\leq |\xi|< 2^{m+1}}  e^{ix\xi}\widehat{f}(\xi)\Vert^q_{L^p(\mathbb{T})}\right)^{\frac{1}{q}}<\infty.
\end{equation}
If $q=\infty,$ $B^r_{p,\infty}(\mathbb{T})$ consists of those functions $f$ satisfying
\begin{equation}\Vert f \Vert_{B^r_{p,\infty}}:=\sup_{m\in\mathbb{N}} 2^{mr}\Vert \sum_{2^m\leq |\xi|< 2^{m+1}}  e^{ix\xi}\widehat{f}(\xi)\Vert_{L^p(\mathbb{T})}<\infty.
\end{equation}
In the case of $p=q=\infty$ and $0<r<1 $ we obtain $B^r_{\infty,\infty}(\mathbb{T})=\Lambda^r(\mathbb{T}),$ these are Banach spaces together with the norm $$\Vert f\Vert_{\Lambda^{r}}=|f|_{\Lambda^r}+\sup_{x\in \mathbb{T}}|f(x)|.$$
Similarly to Besov spaces one defines the Triebel-Lizorkin spaces as follows. If $r\in \mathbb{R},$ $0<p\leq \infty $ $0<q< \infty,$ the Triebel-Lizorkin space $F^{r}_{p,q}(\mathbb{T})$ consists of those functions satisfying
\begin{equation} \Vert f\Vert_{F^r_{p,q}(\mathbb{T})}:=\left\Vert\left(   \sum_{s=0}^{\infty}2^{srq}\left| \sum_{2^{s}\leq |\xi|<2^{s+1}}e^{ix\xi}\widehat{f}(\xi) \right\vert^{q}   \right)^{1/q} \right\Vert_{L^p(\mathbb{T})}
\end{equation}
with a similar modification as in Besov spaces in the case $q=\infty.$ An interesting property regarding Besov spaces and Triebel-Lizorkin spaces is that $B^{r}_{p,p}=F^{r}_{p,p}$ for all $0<p<\infty.$ Now, We introduce some classes of pseudo-differential operators. The periodic H\"ormander class $S^m_{\rho,\delta}(\mathbb{T}^n\times \mathbb{R}^n), \,\, 0\leq \rho,\delta\leq 1,$ consists of those functions $a(x,\xi)$ which are smooth in $(x,\xi)\in \mathbb{T}^n\times \mathbb{R}^n$ and which satisfy toroidal symbols inequalities
\begin{equation}
|\partial^{\beta}_{x}\partial^{\alpha}_{\xi}a(x,\xi)|\leq C_{\alpha,\beta}\langle \xi \rangle^{m-\rho|\alpha|+\delta|\beta|}.
\end{equation}
Symbols in $S^m_{\rho,\delta}(\mathbb{T}^n\times \mathbb{R}^n)$ are symbols in $S^m_{\rho,\delta}(\mathbb{R}^n\times \mathbb{R}^n)$ (see \cite{Ruz}) of order $m$ which are 1-periodic in $x.$
If $a(x,\xi)\in S^{m}_{\rho,\delta}(\mathbb{T}^n\times \mathbb{R}^n),$ the corresponding pseudo-differential operator is defined by
\begin{equation}
a(X,D_{x})u(x)=\int_{\mathbb{T}^n}\int_{\mathbb{R}^n}e^{i2\pi\langle x-y,\xi \rangle}a(x,\xi)u(y)d\xi dy.
\end{equation}
The set $S^m_{\rho,\delta}(\mathbb{T}^n\times \mathbb{Z}^n),\, 0\leq \rho,\delta\leq 1,$ consists of  those functions $a(x, \xi)$ which are smooth in $x$  for all $\xi\in\mathbb{Z}^n$ and which satisfy
\begin{equation}
\forall \alpha,\beta\in\mathbb{N}^n,\exists C_{\alpha,\beta}>0,\,\, |\Delta^{\alpha}_{\xi}\partial^{\beta}_{x}a(x,\xi)|\leq C_{\alpha,\beta}\langle \xi \rangle^{m-\rho|\alpha|+\delta|\beta|}.
 \end{equation}

The operator $\Delta$ is the difference operator defined in \cite{Ruz}. The toroidal operator with symbol $a(x,\xi)$ is defined as
\begin{equation}
a(x,D_{x})u(x)=\sum_{\xi\in\mathbb{Z}^n}e^{i 2\pi\langle x,\xi\rangle}a(x,\xi)\widehat{u}(\xi),\,\, u\in C^{\infty}(\mathbb{T}^n).
\end{equation}

\noindent The corresponding class of operators with symbols in $S^m_{\rho,\delta}(\mathbb{T}^n\times \mathbb{Z}^n)$ (resp. $S^m_{\rho,\delta}(\mathbb{T}^n\times \mathbb{R}^n)$) will be denoted by $\Psi^m_{\rho,\delta}(\mathbb{T}^n\times \mathbb{Z}^n),$ (resp. $\Psi^m_{\rho,\delta}(\mathbb{T}^n\times \mathbb{R}^n)$).
There exists a process to interpolate the second argument of symbols on $\mathbb{T}^n\times \mathbb{Z}^n$ in a smooth way to get a symbol defined on $\mathbb{T}^n\times \mathbb{R}^n.$
\begin{theorem}\label{eq}
Let $0\leq \delta \leq 1,$ $0< \rho\leq 1.$ The symbol $a\in S^m_{\rho,\delta}(\mathbb{T}^n\times \mathbb{Z}^n)$ if only if there exists  a Euclidean symbol $a'\in S^m_{\rho,\delta}(\mathbb{T}^n\times \mathbb{R}^n)$ such that $a=a'|_{\mathbb{T}^n\times \mathbb{Z}^n}.$ Moreover, we have $$ \Psi^m_{\rho,\delta}(\mathbb{T}^n\times \mathbb{Z}^n)=\Psi^m_{\rho,\delta}(\mathbb{T}^n\times \mathbb{R}^n) . $$
\end{theorem}
\begin{proof} The proof can be found in \cite{Ruz}.
\end{proof}
The following results provide some properties about composition and invertibility of periodic pseudo-differential operators. Proofs of these assertions can be found in \cite{Ruz, Ruz-2}.

\begin{theorem}\label{compo}$(\textsl{Composition formula}).$ Let $0\leq \delta<\rho\leq 1.$ The composition $\tau(X,D) \circ \sigma (X,D)$ of two pseudo-differential operators with symbols $\tau \in S^l_{\rho,\delta}(\mathbb{T}^n\times \mathbb{Z}^n)$ and $\sigma \in S^m_{\rho,\delta}(\mathbb{T}^n\times \mathbb{Z}^n)$ is a pseudo-differential operator, and its toroidal  symbol $\psi(x,\xi)$ has the following asymptotic expansion,

\begin{equation}
\psi(x,\xi)\approx \sum_{\gamma \geq 0}\frac{1}{\gamma !}\Delta^{\gamma}_{\xi}\tau(x,\xi)\cdot D^{(\gamma)}_{x}\sigma(x,\xi). 
\end{equation}

\end{theorem}

A pseudo-differential operator $\sigma(x,\xi)\in S^m_{\rho,\delta}$ is called elliptic, if for every $M>0,$
there exists $R>0$ such that $|\sigma(x,\xi)|\geq R\langle \xi \rangle^{m}$ if $|\xi|\geq M.$

\begin{theorem}\label{para}$(\textsl{Parametrix existence}).$ Let $0\leq \delta<\rho\leq 1.$ For every elliptic pseudo-differential operators with symbol $\sigma \in S^m_{\rho,\delta}(\mathbb{T}^n\times \mathbb{Z}^n)$ there exists $\tau \in S^{-m}_{\rho,\delta}(\mathbb{T}^n\times \mathbb{Z}^n)$ such that

\begin{equation}
\sigma(X,D)\circ \tau(X,D)=I+R,\\ \tau(x,D)\circ \sigma(X,D)=I+S,
\end{equation}
where, $S,R$ are pseudo-differential operators with symbols in $S^{-\infty}=\cap_{m}S_{\rho,\delta}^m.$
\end{theorem}

As a consequence of the Proposition 6 in \cite{Eli} and Theorem \ref{eq}, the continuity property of pseudo-differential operators in the H\"older spaces is contained in the following theorems. First we consider the case of operators on $\mathbb{R}^n$ as follow.
\begin{theorem}\label{hol'}
Suppose $\sigma$ is a symbol in $S^m_{1,0}(\mathbb{R}^n\times \mathbb{R}^n).$ Then the operator $\sigma(X,D)$ is a bounded mapping from $\Lambda^{s}(\mathbb{R}^n)$ into $\Lambda^{s-m}(\mathbb{R}^n)$ whenever $m<s\leq 1.$
\end{theorem}

\begin{theorem}\label{hol}
Suppose $\sigma$ is a symbol in $S^m_{1,0}(\mathbb{T}^n\times \mathbb{Z}^n).$ Then the operator $\text{Op}(\sigma)$ is a bounded mapping from $\Lambda^{s}(\mathbb{T}^n)$ into $\Lambda^{s-m}(\mathbb{T}^n)$ whenever $m<s\leq 1.$
\end{theorem}

Our main results are analogues of the Theorem \ref{hol}, but we consider symbols with limited smoothness on the configuration variables $(x,\xi)$.

\section{H\"older-Besov boundedness of periodic operators} \label{mrp}

\subsection{Main results and proofs} In this section we present the proof of our main results. Although all results in this paper are presented for the torus $\mathbb{T}^1$ only, extensions to
the torus $\mathbb{T}^n$ are valid. First, we consider the H\"older-Besov boundedness of periodic H\"older multipliers. Later we extend this result to the case of pseudo-differential operators by considering the Sobolev embedding theorem. This approach was used by Ruzhansky and Wirth \cite{Ruz-3}, (see also \cite{Ruz-T} and \cite{RWT}) in order to get  $L^p$ multiplier theorems for non-invariant pseudo-differential operators on compact Lie groups. We reserve the notaci\'on $A\lesssim B$ if there exists $c>0$ independent of $A$ and $B$ such that $A\leq c\cdot B.$ 

\begin{theorem}\label{main} Let $ \rho\in[0,1]$ and $\sigma(\xi)$ be a symbol satisfying the $\rho$-condition. Then, the corresponding Fourier multiplier $\textnormal{Op}(\sigma):B^{s}_{\infty,\infty}(\mathbb{T})\rightarrow B^{r}_{p,q}(\mathbb{T})$ is a bounded operator for all $r+\frac{1}{2}-\rho<s\leq 1,$ $0< p\leq \infty$ and $0<q< \infty.$ If we assume $r+\frac{1}{2}-\rho\leq s\leq 1,$  we obtain the boundedness of $\textnormal{Op}(\sigma)$ from $B^{s}_{\infty,\infty}$ into $B^{r}_{p,\infty}.$
\end{theorem}
\begin{proof}

Let us consider $f\in C^{\infty}(\mathbb{T}).$ In order to estimate the Besov norm of $\text{Op}(\sigma)f$ we use its dyadic decomposition. First we note that

\begin{equation}\sum_{\xi\in\mathbb{Z}-\{0\}}|\mathscr{F}(\text{Op}(\sigma)f)(\xi)|^2=\sum_{m=0}^{\infty}\sum_{2^m\leq|\xi|<2^{m+1}}|\mathscr{F}(\text{Op}(\sigma)f)(\xi)|^2.
\end{equation}
Now we estimate every dyadic decomposition as follow. If take $h=2\pi/3\cdot 2^m$ and $2^m\leq |\xi|\leq 2^{m+1}$ we have $|e^{-i\xi h}-1|\geq \sqrt{3}.$ Hence we get
\begin{align*} \sum_{2^m\leq|\xi|<2^{m+1}}|\mathscr{F}(\text{Op}(\sigma)f)(\xi)|^2 &\leq \sum_{2^m\leq|\xi|<2^{m+1}}      |e^{-ih\xi}-1|^2 |\mathscr{F}(\text{Op}(\sigma)f)(\xi)|^2\\
&= \sum_{2^m\leq|\xi|<2^{m+1}}      |e^{-ih\xi}-1|^2 |\sigma(\xi)\mathscr{F}(f)(\xi)|^2\\
&\leq \sum_{2^m\leq|\xi|<2^{m+1}}      |e^{-ih\xi}-1|^2 |\xi|^{-2\rho}|\mathscr{F}(f)(\xi)|^2\\
&\lesssim \sum_{2^m\leq|\xi|<2^{m+1}}      |e^{-ih\xi}-1|^2 2^{-2m\rho}|\mathscr{F}(f)(\xi)|^2\\
&\leq 2^{-2m\rho} \sum_{\xi\in\mathbb{Z}}      |e^{-ih\xi}-1|^2 |\mathscr{F}(f)(\xi)|^2.
\end{align*}
On the other hand, Fourier inversion formula guarantees that
\begin{equation}
f(t-h)-f(t)=\sum_{\xi\in\mathbb{Z}}(e^{-i\xi h}-1)(\mathscr{F}f)(\xi)e^{i\xi t}.
\end{equation}
By the Plancherel theorem we conclude that
\begin{align*} \sum_{\xi\in\mathbb{Z}}      |e^{-ih\xi}-1|^2 |\mathscr{F}(f)(\xi)|^2= \Vert f(\cdot -h)-f(\cdot) \Vert^2_{L^{2}(\mathbb{T})}\leq (\frac{2\pi}{3\cdot 2^m})^{2s}\Vert f\Vert^2_{\Lambda^{s}(\mathbb{T})}.
\end{align*}
Hence
\begin{align}\sum_{2^m\leq|\xi|<2^{m+1}}|\mathscr{F}(\text{Op}(\sigma)f)(\xi)|^2 &\leq 2^{-2m\rho}(\frac{2\pi}{3\cdot 2^m})^{2s}\Vert f\Vert^2_{\Lambda^{s}(\mathbb{T})}\\
&\lesssim 2^{-2m(\rho+s)}\Vert f\Vert^2_{\Lambda^{s}}.
\end{align}
By the Cauchy-Schwarz inequality, for all $0<p\leq \infty $ we get
\begin{align*}\left\Vert \sum_{2^m\leq|\xi|<2^{m+1}}\mathscr{F}(\text{Op}(\sigma)f)(\xi)e^{ix\xi}\right\Vert_{L^p(\mathbb{T})} &\leq 2\pi \cdot \sum_{2^m\leq|\xi|<2^{m+1}}|\mathscr{F}(\text{Op}(\sigma)f)(\xi)|\\
&\leq 2\pi \cdot \left(\sum_{2^m\leq|\xi|<2^{m+1}}|\mathscr{F}(\text{Op}(\sigma)f)(\xi)|^2  \right)^{1/2}2^{\frac{1}{2}(m+1)}\\
&\lesssim 2^{-m(\rho+s)+\frac{1}{2}(m+1)} \Vert f\Vert_{\Lambda^{s}}.
\end{align*}
Now, we consider the Besov-norm of $\text{Op}(\sigma)f$ if $0<p,q<\infty:$ in fact, we have
\begin{align*}
\Vert\text{Op}(\sigma)f \Vert_{B^{r}_{p,q}(\mathbb{T})}&:=\left( \sum_{m=0}^{\infty} 2^{mrq}\left\Vert \sum_{2^m\leq |\xi|< 2^{m+1}}  e^{ix\xi}\mathscr{F}((\text{Op}(\sigma)f)(\xi)\right\Vert^q_{L^p(\mathbb{T})}\right)^{\frac{1}{q}}\\
&\leq \left( \sum_{m=0}^{\infty} 2^{mrq} 2^{-mq(\rho+s)+\frac{1}{2}q(m+1)} \Vert f\Vert^q_{\Lambda^{s}}   \right)^{\frac{1}{q}}\\
& \lesssim (\sum_{m=0}^{\infty} 2^{mrq} 2^{-mq(\rho+s)+\frac{1}{2}q(m+1)})^{1/q} \Vert f\Vert_{\Lambda^{s}}.
\end{align*}
From the condition $r+\frac{1}{2}-\rho<s\leq 1$ we obtain
\begin{equation}\sum_{m=0}^{\infty} 2^{mrq} 2^{-mq(\rho+s)+\frac{1}{2}q(m+1)}=2^{\frac{1}{2}}\sum_{m=0}^{\infty} 2^{mq(r-\rho-s+\frac{1}{2})}<\infty.
\end{equation}
Hence $\Vert\text{Op}(\sigma)f \Vert_{B^{r}_{p,q}(\mathbb{T})}\lesssim \Vert f\Vert_{\Lambda^{s}} $ which shows the boundedness of $\text{Op}(\sigma)$ when $q<\infty.$ Now we consider the case $q=\infty.$ In fact, if we assume $r-\rho+\frac{1}{2}\leq s\leq 1,$ we have
\begin{align*}
\Vert\text{Op}(\sigma)f \Vert_{B^{r}_{p,\infty}(\mathbb{T})}&:= \sup_{0\leq m< \infty} 2^{mr}\left\Vert \sum_{2^m\leq |\xi|< 2^{m+1}}  e^{ix\xi}\mathscr{F}((\text{Op}(\sigma)f)(\xi)\right\Vert_{L^p(\mathbb{T})}   \\
&\lesssim \sup_{0\leq m< \infty} 2^{mr} 2^{-m(\rho+s)+\frac{1}{2}(m+1)} \Vert f\Vert_{\Lambda^{s}}  \\
&\lesssim \Vert f\Vert_{\Lambda^{s}}.  
\end{align*}
With above inequality we end the proof.
\end{proof}

\begin{theorem}\label{non} Let us consider $0\leq \rho\leq 1\leq p<\infty,$ $0<q<\infty$ and $r+\frac{1}{2}-\rho<s\leq 1.$ If $\sigma(x,\xi)$ satisfies
\begin{equation}|\partial_{x}^\beta\sigma(x,\xi)|\leq C_{\beta}|\xi|^{-\rho},\,\, |\beta|\leq[1/p]+1,\xi\neq 0,
\end{equation}
then the pseudo-differential operator $\textnormal{Op}(\sigma):B^{s}_{\infty,\infty}(\mathbb{T})\rightarrow B^r_{p,q}(\mathbb{T})$ is a bounded operator.
\end{theorem}
\begin{proof}
Let $f\in C^{\infty}(\mathbb{T}).$ To prove this theorem we write 
\begin{align*} \textnormal{Op}(\sigma)f(x)&=\sum_{\xi\in\mathbb{Z}}e^{ix\xi}\sigma(x,\xi)\widehat{f}(\xi)=\int_{\mathbb{T}}\left(\sum_{\xi\in\mathbb{Z}} e^{i(x-y)\xi}\sigma(x,\xi)\right)f(y)dy\\
&=\int_{\mathbb{T}}\left(\sum_{\xi\in\mathbb{Z}} e^{iy\xi}\sigma(x,\xi)\right)f(x-y)dy.
\end{align*}
Hence, $\text{Op}(\sigma)f(x)=(\varkappa(x,\cdot)\ast f)(x),$ where 
\begin{equation} \varkappa(z,y)=\sum_{\xi\in\mathbb{Z}} e^{iy\xi}\sigma(z,\xi).
\end{equation}
Moreover, if we define $A_{z}f(x)=(\varkappa(z,\cdot)\ast f)(x)$ for every $z\in\mathbb{T},$ we have $$A_{x}f(x)=\text{Op}(\sigma)f(x),\,\,\,x\in\mathbb{T}.$$ For all $0\leq |\beta|\leq [1/p]+1$ we have $\partial^{\beta}_{z}A_{z}f(x)=\text{Op}(\partial_{z}^{\beta}\sigma(z,\cdot))f(x).$
 if $1\leq p<\infty$ we have
\begin{align*}
\Vert \sum_{2^m\leq |\xi|<2^{m+1}} e^{ix\xi}\mathscr{F}(\text{Op}(\sigma)f)(\xi) \Vert^p_{L^{p}}&:=\int_{\mathbb{T}}\left|   \sum_{2^m\leq |\xi|<2^{m+1}} e^{ix\xi}\int_{\mathbb{T}}e^{-iy\xi}\text{Op}(\sigma)f(y)dy \right|^pdx\\
&=\int_{\mathbb{T}}\left|   \sum_{2^m\leq |\xi|<2^{m+1}} e^{ix\xi}\int_{\mathbb{T}}e^{-iy\xi}(A_{y}f)(y)dy \right|^pdx\\ 
&\leq \sup_{z\in\mathbb{T}}\int_{\mathbb{ T}}\left|   \sum_{2^m\leq |\xi|<2^{m+1}} e^{ix\xi}\int_{\mathbb{T}}e^{-iy\xi}(A_{z}f)(y)dy \right|^pdx.\\ 
\end{align*}
By the Sobolev embedding theorem we have

\begin{align*} \sup_{z\in\mathbb{T}}\int_{\mathbb{T}} &\left|   \sum_{2^m\leq |\xi|<2^{m+1}} e^{ix\xi}  \int_{\mathbb{T}}e^{-iy\xi}(A_{z}f)(y)dy \right|^pdx \\
&\lesssim \sum_{|\beta|\leq [1/p]+1}\int_{\mathbb{T}}\int_{\mathbb{T}} \left| \partial^{\beta}_{z}    \sum_{2^m\leq |\xi|<2^{m+1}} e^{ix\xi}  \int_{\mathbb{T}}e^{-iy\xi}(A_{z}f)(y)dy  \right|^pdz\,dx   \\ 
&\leq \sum_{|\beta|\leq [1/p]+1} \int_{\mathbb{T}} \int_{\mathbb{T}}  \left|    \sum_{2^m\leq |\xi|<2^{m+1}} e^{ix\xi}  \mathscr{F}((\partial^{\beta}_{z} A_{z}f))(\xi)  \right| ^pdxdz\\
&= \sum_{|\beta|\leq [1/p]+1} \int_{\mathbb{T}} \Vert    \sum_{2^m\leq |\xi|<2^{m+1}} e^{ix\xi}  \mathscr{F}((\partial^{\beta}_{z} A_{z}f))(\xi)  \Vert^p_{L^p}dz\\
&\lesssim \sup_{|\beta|\leq [1/p]+1}\int_{\mathbb{T}} \Vert    \sum_{2^m\leq |\xi|<2^{m+1}} e^{ix\xi}  \mathscr{F}((\partial^{\beta}_{z} A_{z}f))(\xi)  \Vert^p_{L^p}dz
\end{align*}
Hence,
\begin{align*} \Vert \sum_{2^m\leq |\xi| <2^{m+1}}& e^{ix\xi}\mathscr{F}(\text{Op}(\sigma)f)(\xi) \Vert_{L^{p}}\\
&\lesssim \sup_{|\beta|\leq [\frac{1}{p}]+1} \left(\int_{\mathbb{T}}\Vert    \sum_{2^m\leq |\xi|<2^{m+1}} e^{ix\xi}  \mathscr{F}((\text{Op}(\partial^{\beta}_{z}\sigma(z,\cdot))f))(\xi)  \Vert^{p}_{L^p}dz\right)^{1/p}.
\end{align*}
Thus, considering $0<q<\infty$ we obtain

\begin{align*}
&\Vert\text{Op}(\sigma)f \Vert_{B^{r}_{p,q}(\mathbb{T})}\\
&=\left( \sum_{m=0}^{\infty} 2^{mrq}\left\Vert \sum_{2^m\leq |\xi|< 2^{m+1}}  e^{ix\xi}\mathscr{F}((\text{Op}(\sigma)f)(\xi)\right\Vert^q_{L^p(\mathbb{T})}\right)^{\frac{1}{q}}\\
&\lesssim \left( \sum_{m=0}^{\infty} 2^{mrq} \sup_{|\beta|\leq [\frac{1}{p}]+1} \left(\int_{\mathbb{T}}\Vert    \sum_{2^m\leq |\xi|<2^{m+1}} e^{ix\xi}  \mathscr{F}((\text{Op}(\partial^{\beta}_{z}\sigma(z,\cdot))f))(\xi)  \Vert^{p}_{L^p}dz\right)^{q/p}   \right)^{\frac{1}{q}}\\
&\lesssim \left( \sum_{m=0}^{\infty} 2^{mrq} \sup_{|\beta|\leq [\frac{1}{p}]+1} \left(\int_{\mathbb{T}}\sup_{z\in\mathbb{T}}\Vert    \sum_{2^m\leq |\xi|<2^{m+1}} e^{ix\xi}  \mathscr{F}((\text{Op}(\partial^{\beta}_{z}\sigma(z,\cdot))f))(\xi)  \Vert^{p}_{L^p}dz\right)^{q/p}   \right)^{\frac{1}{q}}\\
&\lesssim \left( \sum_{m=0}^{\infty} 2^{mrq} \sup_{|\beta|\leq [\frac{1}{p}]+1} \left(\sup_{z\in\mathbb{T}}\Vert    \sum_{2^m\leq |\xi|<2^{m+1}} e^{ix\xi}  \mathscr{F}((\text{Op}(\partial^{\beta}_{z}\sigma(z,\cdot))f))(\xi)  \Vert^{p}_{L^p}\right)^{q/p}   \right)^{\frac{1}{q}}.
\end{align*}
Hence, we can write (by using the Fatou's Lemma)
\begin{align*}
\Vert\text{Op}(\sigma)f \Vert_{B^{r}_{p,q}(\mathbb{T})} &\lesssim  \left( \sum_{m=0}^{\infty} 2^{mrq} \sup_{|\beta|\leq [\frac{1}{p}]+1,z\in\mathbb{T}} \Vert    \sum_{2^m\leq |\xi|<2^{m+1}} e^{ix\xi}  \mathscr{F}((\text{Op}(\partial^{\beta}_{z}\sigma(z,\cdot))f))(\xi)  \Vert^q_{L^p}   \right)^{\frac{1}{q}}\\
&\lesssim \sup_{|\beta|\leq [\frac{1}{p}]+1,z\in\mathbb{T}}\Vert\text{Op}(\partial^{\beta}_{z}\sigma(z,\cdot))f \Vert_{B^{r}_{p,q}(\mathbb{T})}  \\
&\leq \left[\sup_{|\beta|\leq [\frac{1}{p}]+1,z\in\mathbb{T}}  \Vert \text{Op}(\partial^{\beta}_{z}\sigma(z,\cdot)) \Vert_{B(\Lambda^s,B^r_{p,q})}\right]\Vert f \Vert_{\Lambda^{s}}.
\end{align*} With the last inequality we end the proof.
\end{proof}

\begin{remark}In order to find connection of H\"older-Besov estimates and $L^p$-estimates, in the next theorem we endowed a H\"older space of degree $0<s<1$ with the norm
\begin{equation}
\Vert f \Vert_{B^{s}_{\infty,\infty,p}}:=\vert f \vert_{\Lambda^s}+\Vert f\Vert_{L^p},
\end{equation}
where $1<p<\infty.$ 
\end{remark}
\begin{theorem}\label{main22} Let $0\leq \rho\leq 1,$ and $\sigma(x,\xi)$ be a measurable function satisfying
\begin{equation}\label{co}
|\partial_{x}^\beta \Delta_{\xi}^{\alpha}\sigma(x,\xi)|\leq C_{\beta}|\xi|^{-\rho-|\alpha|},\,\,\,\,|\beta|\leq [1/p]+1, |\alpha|\leq 2.\,\,\xi\neq 0.
\end{equation}
Then $\textnormal{Op}(\sigma):B^{s}_{\infty,\infty,p}(\mathbb{T})\rightarrow B^{r}_{\infty,\infty,p}(\mathbb{T})$ is a bounded operator for all $r+\frac{1}{2}-\rho\leq s\leq 1.$  
\end{theorem}

\begin{proof}
We use notation as in the proof of Theorem \ref{non}. If we consider the condition \eqref{co}, in particular, we have
\begin{equation}
|\partial_{x}^\beta \sigma(x,\xi)|\leq C_{\beta}|\xi|^{-\rho},\,\,\,\,|\beta|\leq [1/p]+1, |\alpha|\leq 2.\,\,\xi\neq 0.
\end{equation}
So, by  Theorem \ref{non}, for every $z\in \mathbb{T},$ the operator $\partial^{\beta}_{z}A_{z}=\text{Op}(\partial_{z}^{\beta}\sigma(z,\cdot)):B^{s}_{\infty,\infty}(\mathbb{T})\rightarrow B^{r}_{\infty,\infty}(\mathbb{T})$  extends to  bounded operator. Next, we estimate the H\"older-norm of $\text{Op}(\sigma):$ 
\begin{align*}
\vert \text{Op}(\sigma)f\vert_{\Lambda^{r}(\mathbb{T})} &=\sup_{x,h\in\mathbb{T}}\vert \text{Op}(\sigma)f(x-h)- \text{Op}(\sigma)f(x)\vert |h|^{-r}\\
&=\sup_{x,h\in\mathbb{T}}\vert A_{x-h}f(x-h)- A_{x}f(x)\vert |h|^{-r}\\ 
&\leq \sup_{x,h,z\in\mathbb{T}}\vert A_{z}f(x-h)- A_{z}f(x)\vert |h|^{-r}
\end{align*}
By using the Sobolev embedding Theorem we have that
\begin{align*}
\sup_{x,h,z\in\mathbb{T}}\vert A_{z}f(x-h)- A_{z}f( & x)\vert |h|^{-r}\\
&\leq \sup_{x,h\in\mathbb{T}}\sum_{|\beta|\leq [1/p]+1}\Vert\sup_{z\in\mathbb{T}}| \partial_{z}^{\beta}(A_{z}f(x-h)-A_{z}f(x)) |\,\Vert_{L^p(\mathbb{T})}|h|^{-r}\\
&\leq \sup_{x,h\in\mathbb{T}}\sum_{|\beta|\leq [1/p]+1}\sup_{z\in\mathbb{T}}| \partial_{z}^{\beta}(A_{z}f(x-h)-A_{z}f(x)) ||h|^{-r}\\
&\leq \sum_{|\beta|\leq [1/p]+1} \sup_{z\in\mathbb{T}} \sup_{x,h\in\mathbb{T}} | \partial_{z}^{\beta}(A_{z}f(x-h)-A_{z}f(x)) ||h|^{-r}\\
&=\sum_{|\beta|\leq [1/p]+1} \sup_{z\in\mathbb{T}} \vert \partial_{z}^{\beta}(A_{z}f) \vert_{\Lambda^r}\\
&\leq \left[\sum_{|\beta|\leq [1/p]+1} \sup_{z\in\mathbb{T}} \Vert \partial_{z}^{\beta}(A_{z}) \Vert_{B(\Lambda^s,\Lambda^r)}\right]\Vert f \Vert_{\Lambda^{s}}.
\end{align*}
On the other hand, by Theorem 5.2 in \cite{Ruz-3} the operator $\text{Op}(\sigma)$ is a $L^p$-bounded operator for all $1<p<\infty.$ Hence, $\Vert \text{Op}(\sigma)f\Vert_{L^p}\leq C\Vert f\Vert_{L^p}.$ With this in mind, we conclude that
\begin{align}
\Vert \text{Op}(\sigma)f\Vert_{B^{r}_{\infty,\infty,p}(\mathbb{T})}:=\vert \text{Op}(\sigma)f\vert_{\Lambda^{r}(\mathbb{T})}+\Vert \text{Op}(\sigma)f\Vert_{L^p(\mathbb{T})}\lesssim \Vert f \Vert_{B^{s}_{\infty,\infty,p}}.
\end{align}
\end{proof}

\begin{theorem}\label{hardy} Let $0<s<1,$ $2\leq p<\infty$ and $0<q<\infty.$ If $r+1-\frac{2}{p}<\rho\leq 1$ and $\sigma(\xi)$ satisfies the $\rho$-condition, then $\textnormal{Op}(\sigma):B^{s}_{\infty,\infty}(\mathbb{T})\rightarrow B^{r}_{p,q}(\mathbb{T})$ is a bounded linear operator. Moreover, if $r+1-\frac{2}{p}\leq \rho\leq 1,$ then $\textnormal{Op}(\sigma):B^{s}_{\infty,\infty}(\mathbb{T})\rightarrow B^{r}_{p,\infty}(\mathbb{T})$ is a linear  bounded operator.
\end{theorem}
\begin{proof}
First, we recall the Hardy-Littlewood inequality on the torus:
If $2\leq p<\infty$ then
\begin{equation}\label{hll} 
\Vert f \Vert_{L^p(\mathbb{T})}\leq \left(   C_{p}\sum_{\xi\in\mathbb{Z}}(1+|\xi|)^{p-2}|\widehat{f}(\xi)|^p  \right)^{1/p}.
\end{equation}
If we denote by $g_{m}(x)$ the function
\begin{equation}g_{m}(x)= \sum_{2^{m}\leq |\xi|<2^{m+1}}  e^{ix\xi}\mathscr{F}(\text{Op}(\sigma)f)(\xi),
\end{equation}
then $g_{m}=\mathscr{F}^{-1}[\chi_{\{2^{m}\leq |\xi|<2^{m+1}\} }\cdot \mathscr{F}(\text{Op}(\sigma)f)(\cdot)].$ By \eqref{hll} we have,
$$ \Vert g_{m} \Vert_{L^p(\mathbb{T})}\leq \left[C_p\cdot\sum_{2^m\leq |\xi|<2^{m+1}}(1+|\xi|)^{p-2}| \mathscr{F}(\text{Op}(\sigma)f)(\xi) |^{p} \right]^{1/p} ,$$
therefore, for $0<q<\infty$ we obtain
\begin{align*} \sum^{\infty}_{m=0}2^{mrq}\Vert \sum_{2^{m}\leq |\xi|<2^{m+1}} & e^{ix\xi}\mathscr{F}(\text{Op}(\sigma)f)(\xi) \Vert^{q}_{L^p(\mathbb{T})}\\
&\lesssim  \sum^{\infty}_{m=0}2^{mrq}\left[ \sum_{2^m\leq |\xi|<2^{m+1}}(1+|\xi|)^{p-2}|^p \mathscr{F}(\text{Op}(\sigma)f)(\xi) | \right]^{q/p}\\
&\lesssim \sum^{\infty}_{m=0}2^{mrq}\left[ \sum_{2^m\leq |\xi|<2^{m+1}}2^{m(p-2)}| \sigma(\xi)\mathscr{F}(f)(\xi) |^p \right]^{q/p}\\
&\lesssim \sum^{\infty}_{m=0}2^{mrq}\left[ 2^{m(p-2)}{2^{-m\rho p}} \right]^{q/p} \Vert \widehat{f}\Vert^{q}_{L^p(\mathbb{Z})}.\\
\end{align*}
Considering that $\Vert \widehat{f}\Vert_{L^p(\mathbb{Z})}\lesssim \Vert f \Vert_{\Lambda^s}$ for every $0<s<1$ and $2\leq p<\infty$ we get
\begin{equation}\sum^{\infty}_{m=0}2^{mrq}\Vert \sum_{2^{m}\leq |\xi|<2^{m+1}}  e^{ix\xi}\mathscr{F}(\text{Op}(\sigma)f)(\xi) \Vert^q_{L^p(\mathbb{T})}\lesssim \sum^{\infty}_{m=0}2^{mrq+m(p-2)\frac{q}{p}-m\rho q}  \Vert f\Vert^{q}_{\Lambda^s}.
\end{equation}
Since $r+1-\frac{2}{p}<\rho$ we get $$ C=\sum^{\infty}_{m=0}2^{mrq+m(p-2)\frac{q}{p}-m\rho q} <\infty. $$
So, $\Vert \text{Op}(\sigma) \Vert_{B^r_{p,q}}\lesssim \Vert f\Vert_{\Lambda^s}.$ Hence, we conclude the boundedness of $$\text{Op}(\sigma):B^{s}_{\infty,\infty}(\mathbb{T})\rightarrow B^{r}_{p,q}(\mathbb{T}).$$ The proof of the boundedness of $\text{Op}(\sigma)$ when $q=\infty$ is analogue.
\end{proof}

We extend Theorem \ref{hardy} to case of non-invariant periodic operators  as follows:

\begin{theorem} Let us consider $2\leq p<\infty,$ $0<q<\infty$ and $r+1-\frac{2}{p}<\rho\leq 1.$ Let $\sigma(x,\xi)$ be a symbol satisfying
\begin{equation}|\partial^{\beta}_{x}\sigma(x,\xi)|\leq C_{\beta}|\xi|^{-\rho}, \,\,|\beta|\leq[1/p]+1,\,\,\xi\neq 0.
\end{equation} Then $\textnormal{Op}(\sigma)$ is a bounded operator from $B^{s}_{\infty,\infty}$ into $B^{r}_{p,q}.$
\end{theorem}

\begin{proof} The proof of this theorem is similar to the proof of Theorem \ref{non}.
\end{proof}

Now, we prove results concerning H\"older-Triebel boundedness of Fourier multipliers.

\begin{theorem}\label{trie} Let us consider $0<q<\infty,$ $0<p\leq \infty,$ $r<\rho\leq 1$ and $\frac{1}{2}<s\leq 1.$ Then $\textnormal{Op}(\sigma):B^{s}_{\infty,\infty}(\mathbb{T})\rightarrow F^{r}_{p,q}(\mathbb{T})$ is a bounded operator if we consider that $\sigma(\xi)$ satisfies the $\rho$-condition. If we assume $r\leq \rho\leq 1,$ $\frac{1}{2}<s\leq1$ and $q=\infty$ then $\textnormal{Op}(\sigma):B^{s}_{\infty,\infty}(\mathbb{T})\rightarrow F^{r}_{p,\infty}(\mathbb{T}),$  is a bounded operator.
\end{theorem}
\begin{proof}
First we consider the case of $0<q<\infty,$ $0<p\leq \infty$ and $r<\rho.$ By the definition of Triebel-Lizorkin norm, we have
\begin{align*}
\Vert \text{Op}(\sigma)f\Vert_{F^r_{p,q}(\mathbb{T})}&:=\left\Vert\left(   \sum_{m=0}^{\infty}2^{mrq}\left| \sum_{2^{m}\leq |\xi|<2^{m+1}}e^{ix\xi}\sigma(\xi)\widehat{f}(\xi) \right\vert^{q}   \right)^{1/q} \right\Vert_{L^p(\mathbb{T})}\\
&\leq \left\Vert\left(   \sum_{m=0}^{\infty}2^{mrq}\left| \sum_{2^{m}\leq |\xi|<2^{m+1}}|\sigma(\xi)\widehat{f}(\xi) |\right\vert^{q}   \right)^{1/q} \right\Vert_{L^p(\mathbb{T})}\\
&= \left(   \sum_{m=0}^{\infty}2^{mrq}\left| \sum_{2^{m}\leq |\xi|<2^{m+1}}|\sigma(\xi)\widehat{f}(\xi) |\right\vert^{q}   \right)^{1/q} \\
&\leq \left(   \sum_{m=0}^{\infty}\left| \sum_{\xi\in\mathbb{Z}}2^{mr-m\rho}|\widehat{f}(\xi) |\right\vert^{q}   \right)^{1/q} .
\end{align*}
By using the Minkowski integral inequality (discrete version) we have
$$ \left(   \sum_{m=0}^{\infty}\left| \sum_{\xi\in\mathbb{Z}}2^{mr-m\rho}|\widehat{f}(\xi) |\right\vert^{q}   \right)^{1/q} \leq \sum_{\xi\in\mathbb{Z}}|\widehat{f}(\xi)|\left[\sum_{m=0}  2^{qm(r-\rho)}\right]^{1/q} . $$ 
From the condition $r<\rho$ and by using the Bernstein theorem (i.e $\Vert \widehat{f} \Vert_{L^1(\mathbb{Z})}\lesssim \Vert f\Vert_{\Lambda^s},\frac{1}{2}<s\leq 1$) we have
$$ \Vert \text{Op}(\sigma)f\Vert_{F^r_{p,q}(\mathbb{T})}\lesssim \Vert f \Vert_{\Lambda^s}. $$
If $q=\infty$ and $r\leq \rho$ we observe that
\begin{align*}
\Vert \text{Op}(\sigma)f\Vert_{F^r_{p,q}(\mathbb{T})}&:= \sup_{m\in\mathbb{N}}2^{mr}\left\Vert   \left| \sum_{2^{m}\leq |\xi|<2^{m+1}}e^{ix\xi}\sigma(\xi)\widehat{f}(\xi)    \right| \right\Vert_{L^p(\mathbb{T})}\\
&\leq \sup_{m\in\mathbb{N}}2^{mr}\sum_{2^{m}\leq |\xi|<2^{m+1}}|m(\xi)||\widehat{f}(\xi)|\\
&\leq \sup_{m\in\mathbb{N}}2^{m(r-\rho)}\sum_{\xi\in\mathbb{Z}}|\widehat{f}(\xi)|\\
&\lesssim \Vert f\Vert_{\Lambda^s}.
\end{align*}
\end{proof}
In order to get boundedness from H\"older into Triebel-Lizorkin spaces, we present the following lemma which is a generalization of the Bernstein Theorem. (See \cite{b, blo}). We recall the equivalence $\Lambda^{s}(\mathbb{T})\equiv B^{s}_{\infty,\infty}(\mathbb{T})$ for the H\"older space of order $s.$ 
\begin{lemma}\label{general}
 Let $2/3<p\leq 2$ and let $s_p=1/p-1/2.$ Then, the Fourier transform $f\mapsto \mathscr{F}f$ from $\Lambda^s(\mathbb{T})$ into $L^{p}(\mathbb{T})$ is a bounded operator for all $s,$ $s_p<s<1.$  
\end{lemma}
\begin{theorem}\label{trie2} Let $0\leq \rho\leq 1<\alpha\leq 2,$ $s_{\alpha}=\frac{1}{\alpha}-\frac{1}{2},$ and $r+1-\frac{1}{\alpha}<\rho\leq 1.$ If $\sigma(\xi)$ satisfies the $\rho$-condition, then $\textnormal{Op}(\sigma):B^{s}_{\infty,\infty}(\mathbb{T})\rightarrow F^{r}_{p,q}(\mathbb{T})$ is a bounded operator for all $0<p\leq \infty,$ $0<q<\infty$ and $s_{\alpha}<s<1.$ Moreover, if $r+1-\frac{1}{\alpha}\leq\rho,$ the operator $\textnormal{Op}(\sigma):B^{s}_{\infty,\infty}(\mathbb{T})\rightarrow F^{r}_{p,\infty}(\mathbb{T})$ is bounded.
\end{theorem}
\begin{proof}
From the proof of Theorem \ref{trie} we have
$$ \Vert \text{Op}(\sigma)f\Vert_{F^r_{p,q}(\mathbb{T})}\leq   \left(   \sum_{m=0}^{\infty}2^{mrq}\left| \sum_{2^{m}\leq |\xi|<2^{m+1}}|\sigma(\xi)\widehat{f}(\xi) |\right\vert^{q}   \right)^{1/q}  .$$
On the other hand, if $1<\alpha \leq 2$ and $1/\alpha+1/\alpha'=1,$ by using the H\"older inequality we obtain
\begin{align*}
& \left(   \sum_{m=0}^{\infty}2^{mrq}  \left| \sum_{2^{m}\leq |\xi|<2^{m+1}}|\sigma(\xi)\widehat{f}(\xi) |\right\vert^{q}   \right)^{1/q}\\
&\leq  \left(   \sum_{m=0}^{\infty} 2^{mq(r-\rho)}\left| \left(\sum_{\xi\in\mathbb{Z}}|\widehat{f}(\xi)|^{\alpha}\right)^{\frac{1}{\alpha}}   \left(\sum_{ 2^{m}\leq |\xi|<2^{m+1} } 1\right)^{\frac{1}{\alpha'}}          \right\vert^{q}   \right)^{1/q} \\
&\leq \left(   \sum_{m=0}^{\infty} 2^{mq(r-\rho)}\left|    (2^{\frac{m+1}{\alpha'}})          \right\vert^{q}   \right)^{1/q}\Vert \widehat{f} \Vert_{L^{\alpha}(\mathbb{Z})}\\
&\lesssim \left(   \sum_{m=0}^{\infty} 2^{mq(r-\rho+\frac{1}{\alpha'})}   \right)^{1/q}\Vert \widehat{f} \Vert_{L^{\alpha}(\mathbb{Z})}.
\end{align*}
From Lemma \ref{general} and the condition $r+1-\frac{1}{\alpha}<\rho\leq 1$ we claim

$$ \Vert \text{Op}(\sigma)f\Vert_{F^r_{p,q}(\mathbb{T})}\lesssim \Vert f \Vert_{\Lambda^s}$$
for all $s_{\alpha}<s<1.$
By a similar argument, we may prove
$$ \Vert \text{Op}(\sigma)f\Vert_{F^r_{p,\infty}(\mathbb{T})}\lesssim \sup_{m\in\mathbb{N}}2^{m(r-\rho+1/\alpha')} \Vert f\Vert_{\Lambda^s}. $$
Hence,
$$ \Vert \text{Op}(\sigma)f\Vert_{F^r_{p,\infty}(\mathbb{T})}\lesssim \Vert f \Vert_{\Lambda^s}$$
for all $s_{\alpha}<s<1,$ $0<p\leq \infty$ and $r+1-\frac{1}{\alpha}\leq \rho.$ 
\end{proof}

\begin{theorem}\label{HTC} Let us consider the periodic pseudo-differential operator $\textnormal{Op}(\sigma)$ with the symbol $\sigma(x,\xi)$ satisfying 
$$  \vert\partial_{x}^{\alpha} \sigma(x,\xi) \vert\leq C_{\alpha}|\xi|^{-\rho},\,\,|\alpha|\leq [1/q]+1,\xi\neq 0.$$
If $1\leq q<\infty,$ $0<p\leq \infty,$ $r<\rho\leq 1$ and $\frac{1}{2}<s\leq 1,$ then $\textnormal{Op}(\sigma)$ is bounded from $\Lambda^s$ into $F^{r}_{p,q}.$ Also, if we assume
 $0\leq \rho\leq 1<\alpha\leq 2,$ $s_{\alpha}=\frac{1}{\alpha}-\frac{1}{2},$ $0<p\leq \infty$ and $r+1-\frac{1}{\alpha}<\rho\leq 1,$ $\textnormal{Op}(\sigma):\Lambda^s\rightarrow F^{r}_{p,q}$ is a bounded linear operator for all $s_{\alpha}<s<1.$ 
\end{theorem}
\begin{proof}
If $1\leq q<\infty,$ by the Sobolev embedding theorem we write,

\begin{align*}
\left| \sum_{2^{m}\leq |\xi|<2^{m+1}}e^{ix\xi}\mathscr{F}( \text{Op}(\sigma)f )(\xi) \right\vert^{q} &\leq \sup_{z\in\mathbb{T}}\left| \sum_{2^{m}\leq |\xi|<2^{m+1}}e^{ix\xi}\mathscr{F}( \text{Op}(\sigma(z,\cdot))f )(\xi) \right\vert^{q}\\
&\lesssim \sum_{|\alpha|\leq [1/q]+1}\int_{\mathbb{T}}\vert  \sum_{2^{m}\leq |\xi|<2^{m+1}}e^{ix\xi}\mathscr{F}( \text{Op}(\partial_{z}^{\alpha}\sigma(z,\cdot))f )(\xi)   \vert^qdz\\
&\lesssim \sup_{|\alpha|\leq [1/q]+1}\int_{\mathbb{T}}\vert   \sum_{2^{m}\leq |\xi|<2^{m+1}}e^{ix\xi}\mathscr{F}( \text{Op}(\partial_{z}^{\alpha}\sigma(z,\cdot))f )(\xi)   \vert^qdz
\end{align*}

From this inequality we deduce that
\begin{align*}
&\Vert \text{Op}(\sigma)f\Vert_{F^r_{p,q}(\mathbb{T})}\\
&=\left\Vert\left(   \sum_{m=0}^{\infty}2^{mrq}\left| \sum_{2^{m}\leq |\xi|<2^{m+1}}e^{ix\xi}\mathscr{F}^{-1}[\sigma(x,\xi)\widehat{f}\,](\xi) \right\vert^{q}   \right)^{1/q} \right\Vert_{L^p(\mathbb{T})}\\
&\lesssim \left\Vert\left(   \sum_{m=0}^{\infty}2^{mrq}\int_{\mathbb{T}} \sup_{|\alpha|\leq [1/q]+1}\vert   \sum_{2^{m}\leq |\xi|<2^{m+1}}e^{ix\xi}\mathscr{F}( \text{Op}(\partial_{z}^{\alpha}\sigma(z,\cdot))f )(\xi)   \vert^qdz  \right)^{1/q} \right\Vert_{L^p(\mathbb{T})}\\
&\lesssim \left\Vert\left( \sum_{m=0}^{\infty}2^{mrq}  \sup_{|\alpha|\leq [1/q]+1
}\sup_{z\in\mathbb{T}}   \vert   \sum_{2^{m}\leq |\xi|<2^{m+1}}e^{ix\xi}\mathscr{F}( \text{Op}(\partial_{z}^{\alpha}\sigma(z,\cdot))f )(\xi)   \vert^q  \right)^{1/q} \right\Vert_{L^p(\mathbb{T})}\\
&\lesssim \sup_{|\alpha|\leq [1/q]+1
}\sup_{z\in\mathbb{T}}  \left\Vert\left( \sum_{m=0}^{\infty}2^{mrq}   \vert   \sum_{2^{m}\leq |\xi|<2^{m+1}}e^{ix\xi}\mathscr{F}( \text{Op}(\partial_{z}^{\alpha}\sigma(z,\cdot))f )(\xi)   \vert^q  \right)^{1/q} \right\Vert_{L^p(\mathbb{T})}.
\end{align*}
Hence we get  (by using the Fatou's Lemma)
\begin{align*}
&\Vert \text{Op}(\sigma)f\Vert_{F^r_{p,q}(\mathbb{T})}\\&\lesssim \sup_{|\alpha|\leq [1/q]+1,z\in\mathbb{T}} \left\Vert\left(   \sum_{m=0}^{\infty}2^{mrq} \vert   \sum_{2^{m}\leq |\xi|<2^{m+1}}e^{ix\xi}\mathscr{F}( \text{Op}(\partial_{z}^{\alpha}\sigma(z,\cdot))f )(\xi)   \vert^q \right)^{1/q} \right\Vert_{L^p(\mathbb{T})}\\
&\lesssim \sup_{|\alpha|\leq [1/q]+1,z\in\mathbb{T}} \Vert \text{Op}(\partial^{\alpha}_{z}\sigma(z,\cdot))f\Vert_{F^r_{p,q}(\mathbb{T})} \\
&\leq \sup_{|\alpha|\leq [1/q]+1,z\in\mathbb{T}} \Vert \text{Op}(\partial^{\alpha}_{z}\sigma(z,\cdot))\Vert_{B(\Lambda^s,F^r_{p,q})}\Vert f \Vert_{\Lambda^s}. \\
\end{align*}
So, by the  last inequality, Theorem \ref{trie} and Theorem \ref{trie2}, we deduce the Boundedness of $\text{Op}(\sigma(\cdot,\cdot))$ from $\Lambda^s$ into $F^{r}_{p,q}$ in the following cases: 
\begin{itemize}
\item $1\leq q<\infty,$ $ 0<p\leq \infty,$  $r<\rho, \frac{1}{2}<s\leq 1.$
\item $0\leq \rho\leq 1,$ $1<\alpha\leq 2,$ $0<p\leq \infty,$ $r+1-\frac{1}{\alpha}<\rho$ and $s_{\alpha}<s<1.$
\end{itemize}
\end{proof}

\begin{theorem}\label{min} Let us consider $\textnormal{Op}(\sigma)$ be a Fourier multiplier with symbol satisfying the $\rho-$condition. Then $\textnormal{Op}(\sigma):B^{s}_{\infty,\infty}(\mathbb{T})\rightarrow F^{r}_{p,q}(\mathbb{T}),$ $0< p\leq \infty,$ $1<q<\infty$ is a bounded operator if $r+\frac{1}{2}-\rho<s\leq 1.$ If  $r+\frac{1}{2}-\rho\leq s\leq 1,$ then $\textnormal{Op}(\sigma):B^{s}_{\infty,\infty}(\mathbb{T})\rightarrow F^{r}_{p,\infty}(\mathbb{T})$ is continuous.  
\end{theorem}
\begin{proof}
From the proof of Theorem \ref{main} we have
\begin{align}\sum_{2^m\leq|\xi|<2^{m+1}}|\mathscr{F}(\text{Op}(\sigma)f)(\xi)|^2 &\leq 2^{-2m\rho}(\frac{2\pi}{3\cdot 2^m})^{2s}\Vert f\Vert^2_{\Lambda^{s}(\mathbb{T})}\\
&\lesssim 2^{-2m(\rho+s)}\Vert f\Vert^2_{\Lambda^{s}}.
\end{align}
Hence
\begin{align*}
& \left(   \sum_{m=0}^{\infty}2^{mrq}  \left|  \sum_{2^{m}\leq |\xi|<2^{m+1}}|\sigma(\xi)\widehat{f}(\xi) |\right\vert^{q}   \right)^{1/q}\\
& \leq \left(   \sum_{m=0}^{\infty}2^{mrq}  \left|  \sum_{2^{m}\leq |\xi|<2^{m+1}}|\sigma(\xi)\widehat{f}(\xi) |^{2}\right\vert^{q/2}2^{q(m+1)/2}   \right)^{1/q}\\
& \lesssim\left(   \sum_{m=0}^{\infty}2^{mrq} 2^{-mq(\rho+s)}2^{q(m+1)/2}   \right)^{1/q}\Vert f\Vert_{\Lambda^s}.\\
\end{align*}
From the condition $r-\rho+\frac{1}{2}<s$ we deduce the boundedness of $\text{Op}(\sigma),$ in fact
$$ \Vert \text{Op}(\sigma)f\Vert_{F^r_{p,q}(\mathbb{T})}\leq \left(   \sum_{m=0}^{\infty}2^{mrq} 2^{-mq(\rho+s)+q(m+1)/2}   \right)^{1/q}  \Vert f \Vert_{\Lambda^s}.$$
A similar proof is valid for $q=\infty$ and $r-\rho+\frac{1}{2}\leq s.$
\end{proof}
\begin{remark} We observe that similar extensions that we give here of the Theorem \ref{main} to the  non-invariant case of pseudo-differential operators can be obtained if in place of $B^{r}_{p,q}(\mathbb{T})$ we write $F^{r}_{p,q}(\mathbb{T})$.
\end{remark}

We end this section with the following theorem on boundedness of periodic pseudo-differential operators on H\"older spaces.
\begin{theorem}\label{13}
Let $2/3< p\leq 2,$ $s_{p}=1/p-1/2,$ $0<r<1$ and $s_{p}<s<1.$ If $r+\frac{1}{q}\leq \rho$ and $\sigma(\xi)$ satisfies the $\rho$-condition, then  $\textnormal{Op}(\sigma):B^{s}_{\infty,\infty}(\mathbb{T})\rightarrow B^{r}_{\infty,\infty}(\mathbb{T})$ is a bounded Fourier multiplier.
\end{theorem}
\begin{proof}
First we consider the case where $\sigma$ depends only on the Fourier variable $\xi.$ So we get for $s\geq 0$
\begin{align*}
\Vert  \mathrm{Op}(\sigma)f\Vert_{B^{r}_{\infty,\infty}} &=\sup_{s\in\mathbb{N}} 2^{sr}\Vert \sum_{2^{s}\leq |\xi|<2^{s+1}}e^{ix\xi}\mathscr{F}(\mathrm{Op}(\sigma)f)(\xi) \Vert_{L^{\infty}(\mathbb{T})}\\
&\leq \sup_{s\in\mathbb{N}} 2^{sr} \sum_{2^{s}\leq |\xi|<2^{s+1}}|\mathscr{F}(\mathrm{Op}(\sigma)f)(\xi)| \\
&=\sup_{s\in\mathbb{N}}2^{sr} \sum_{2^{s}\leq |\xi|<2^{s+1}}|\widehat{f}(\xi)||\sigma(\xi)| 
\end{align*}
By H\"older inequality we obtain
\begin{align*}
\Vert  \mathrm{Op}(\sigma)f\Vert_{B^{r}_{\infty,\infty}} &\lesssim \sup_{s\in\mathbb{N}}\left(\sum_{2^{s}\leq |\xi|<2^{s+1}} |\hat{f}(\xi)|^{p}\right)^{1/p}\left( \sum_{2^{s}\leq |\xi|<2^{s+1}}|\sigma(\xi)|^{q}2^{srq} \right)^{1/q}\\
&\lesssim \sup_{s\in\mathbb{N}} \Vert \widehat{f} \Vert_{L^{p}(\mathbb{Z})}\left( \sum_{2^{s}\leq |\xi|<2^{s+1}}\langle \xi \rangle^{-\rho q}2^{srq} \right)^{1/q}\\
&\lesssim \sup_{s\in\mathbb{N}} \Vert \widehat{f} \Vert_{L^{p}(\mathbb{Z})}\left( \sum_{2^{s}\leq |\xi|<2^{s+1}}2^{-s\rho q}2^{srq} \right)^{1/q}\\
&\lesssim \sup_{s\in\mathbb{N}} \Vert \widehat{f} \Vert_{L^{p}(\mathbb{Z})}\left( 2^{sq(r-\rho)+s} \right)^{1/q}.\\
\end{align*}
By Lemma \ref{general} we have $\Vert \hat{f}(\xi)\Vert_{L^{p}(\mathbb{Z})}\lesssim \Vert f\Vert_{B^{s}_{\infty,\infty}}$ for $s_{p}<s<1.$ Since $r+\frac{1}{q}\leq \rho$ we get,  
$$ \Vert  \mathrm{Op}(\sigma)f\Vert_{  B^{r}_{\infty,\infty}  } \lesssim \Vert f \Vert_{B^{s}_{\infty,\infty}}  $$
which proves the boundedness of $\textrm{Op}(\sigma).$
\end{proof}

\subsection{Remarks and examples} There exists a connection between the $L^{p}$ boundedness of Fourier multipliers on compact Lie groups and its continuity on Besov spaces. This fact was proved by the author in Theorem 1.2 of \cite{Duvan4}. In fact, the Lie group structure of the torus $\mathbb{T}$ implies that every periodic Fourier multiplier bounded from $L^{p_{1}}$ into $L^{p_{2}}$ is bounded from $B^{r}_{p_1,q}$ into $B^{r}_{p_2,q},$ $r\in\mathbb{R}$ and $0<q\leq \infty.$ Since, in general, the boundedness of Fourier multipliers satisfying the $\rho$-condition ---or pseudo-differential operators with symbols satisfying H\"ormander conditions but with limited regularity--- fails for $p_{i}=\infty$, we have concentrate  our attention to this case in the preceding subsection, in order to give boundedness of multipliers ---and of pseudo-differential operators--- in H\"older spaces $\Lambda^{r}\equiv B^{r}_{\infty,\infty}.$ 
\begin{remark}
With the discussion above in mind, periodic Fourier multipliers with symbol $\sigma(\xi)$ satisfying the variational Marcinkiewicz condition: 
\begin{equation}
\Vert \sigma \Vert_{L^{\infty}(\mathbb{Z})}+\sup_{j\geq 0}\sum_{2^j\leq |\xi|\leq 2^{j+1}}|\sigma(\xi+1)-\sigma(\xi)|<\infty,
\end{equation}
are bounded from $L^{p}(\mathbb{T}),$ $1<p<\infty$ and hence these operators are bounded on every Besov space $B^{r}_{p,q}(\mathbb{T})$ but, its boundedness on H\"older spaces fails (Corollary 4.3 of \cite{Ar}). It is important to mention that every Fourier multiplier satisfying the $\rho$-condition \eqref{rho2} with $0<\rho\leq 1$ also satisfies the variational Marcinkiewicz condition and, as a consequence, these operators are bounded on $L^{p},$ $1<p<\infty$ and on every Besov space $B^{r}_{p,q},$ $r\in\mathbb{R}$, $1<p<\infty$ and $0<q\leq \infty.$
\end{remark}
\begin{remark}
Theorems \ref{main}, \ref{hardy}, \ref{trie}, \ref{trie2}, \ref{min} and \ref{13} give boundedness of Fourier multipliers from H\"older spaces $B^{s}_{\infty,\infty}$ into Besov spaces $B^{r}_{p,q}$ or spaces of Triebel Lizorkin $F^{r}_{p,q}$. Theorem \ref{main} shows a dependence of the parameters $\rho,r$ and $s.$ Nevertheless, as a consequence of the Hardy-Littlewood inequality, Theorem \ref{hardy} relaxes this type of conditions for $2\leq p<\infty$ by imposing restrictions on $\rho, r$ and $p.$ On the other hand, for $\frac{2}{3}<p\leq 2,$ Theorem \ref{13} only consider a dependence on the parameters $r,\rho$ and $q.$ These theorems have been proved by using non-trivial modifications of the proof of the Bernstein Theorem \cite{b}.  Theorems \ref{main22}, \ref{non} and \ref{HTC} have been proved using the Sobolev embedding theorem as a fundamental tool.
\end{remark}
\begin{remark}\label{rhodelta}
Notice that the results of this section illustrate a very important connection between  $L^p$ boundedness and H\"{o}lder boundedness. Indeed, by observing the proof of Theorem \ref{main22}, the condition on the symbol
\begin{equation}
|\partial_{x}^\beta \sigma(x,\xi)|\leq C_{\beta}|\xi|^{-\rho},\,\,\,\,|\beta|\leq [1/p]+1.\,\,\xi\neq 0,
\end{equation}
guarantees the boundedness of $\textnormal{Op}(\sigma),$ from $B^{s}_{\infty,\infty}$ into $B^{r}_{\infty,\infty}$ for $r-s+\frac{1}{2}\leq \rho\leq 1.$ For the $L^p$ boundedness, $1<p<\infty$ of $\textnormal{Op}(\sigma)$ (see Theorem 3.7 of \cite{DR4}) it is sufficient to consider the following condition
\begin{equation}\label{holderbesov}
|\Delta_{\xi}^{\alpha}\partial_{x}^\beta \sigma(x,\xi)|\leq C_{\beta}|\xi|^{-\rho-\tilde{\rho}|\alpha|},\,\,\,\,|\beta|\leq [1/p]+1, |\alpha|\leq 2.\,\,\xi\neq 0,
\end{equation}
for $\rho=2(1-\tilde{\rho})|\frac{1}{p}-\frac{1}{2}|,$ $0\leq \tilde{\rho}\leq 1.$ Thus, if we consider the inequality \eqref{holderbesov}, with 
$$ r-s+\frac{1}{2}\leq \rho=2(1-\tilde{\rho})\left|\frac{1}{p}-\frac{1}{2}\right|\leq 1, $$
we obtain the boundedness of $\textnormal{Op}(\sigma)$ from $B^{s}_{\infty,\infty,p}$ into $B^{r}_{\infty,\infty,p}.$
\end{remark}
We end this  section with the following examples on operators satisfying the $\rho$-condition and on elliptic regularity in H\"older spaces.
\begin{example}
Let $X$ be a left-invariant real vector field on the torus $\mathbb{T}.$ By Corollary 2.7 of \cite{Ruz-3}, there exists an exceptional set $\mathscr{C}\subset i\mathbb{R},$ such that for all $c\notin \mathscr{C}, $ the operator $X+c$ is invertible with inverse satisfying the $\rho$-condition with $\rho=0.$ By Theorem \ref{main}, we have for $r+\frac{1}{2}<s\leq1,$ $0<p\leq \infty$ and $0<q<\infty:$
\begin{equation}
\Vert f \Vert_{B^{r}_{p,q}}\leq C\Vert (X+c)f\Vert_{B^{s}_{\infty,\infty}}.
\end{equation}
On the other hand, if we consider $r+\frac{1}{2}\leq s\leq1,$ we obtain the estimate,
\begin{equation}
\Vert f \Vert_{B^{r}_{p,\infty}}\leq C\Vert (X+c)f\Vert_{B^{s}_{\infty,\infty}}. 
\end{equation}
Analogous estimates may be obtained if we apply Theorem \ref{hardy}, Theorem \ref{trie} or Theorem \ref{13}. Similar results also can be considered if we replace $X$ by the partial Riesz transform $\mathfrak{R}=(-\mathcal{L}_{\mathbb{T}})^{-\frac{1}{2}}\circ X,$  of some negative power of the Laplace operator $\mathcal{L}_{\mathbb{T}}$.
\end{example}
\begin{example}
Let $0<r<1$ and $f\in B^{r}_{\infty,\infty}(\mathbb{T}) .$ Consider the toroidal pseudo-differential problem

\begin{equation}  
     \textnormal{Op}(\sigma)u =f, \\ \vspace{0.2cm}    
\end{equation}
where $\textnormal{Op}(\sigma)$ is an elliptic operator with symbol $\sigma(x,\xi)\in S^m_{\rho,\delta},$ $m>0,$ (in particular, $\textnormal{Op}(\sigma)$ can be an elliptic differential operator of the form $\sum_{0\leq i\leq m}a_{i}(x)\partial^{i}_{x},\,$ $m\geq 1$). By the existence of parametrices for elliptic operators, (see Theorems \ref{compo} and \ref{para}), there exists $q\in S^{-m}_{\rho,\delta}$ and $r\in S^{-\infty}$ such that
 \begin{equation}\textnormal{Op}(q)\circ \textnormal{Op}(\sigma)=I+\textnormal{Op}(r).
 \end{equation}
Therefore, $ \textnormal{Op}(q)\circ \textnormal{Op}(\sigma)u=u+\textnormal{Op}(r)u=\textnormal{Op}(q)f.  $ By using Theorem \ref{main22}, we have for $0< r\leq s+m-\frac{1}{2}, $ $0<s<1,$ 
$$  \Vert \textnormal{Op}(q)f\Vert _{B^{r}_{\infty,\infty,p}}\leq C\Vert f\Vert _{B^{s}_{\infty,\infty,p}},$$ and considering that the operator $\textnormal{Op}(r)$ is a smoothing operator we get $u\in B^{r}_{\infty,\infty,p}.$ In conclusion, under the pseudo-differential problem considered, if $f\in B^{s}_{\infty,\infty,p}(\mathbb{T}) $ then $u\in B^{r}_{\infty,\infty,p}(\mathbb{T}).$ A similar {\textit{a priori estimate}} can be obtained if we consider Theorem \ref{HTC}.
\end{example}
\noindent \textbf{Acknowledgments.} 
I would like to thank the anonymous referee for his remarks which helped to improve the manuscript. The author is indebted with Alexander Cardona for helpful comments on an earlier draft of this paper. This project was partially supported by Universidad de los Andes, Mathematics Department.

\bibliographystyle{amsplain}

\end{document}